\documentclass[12pt,reqno]{amsart}
\usepackage{indentfirst, amssymb, amsmath, amsthm, mathrsfs, setspace, indentfirst, enumerate,  mathrsfs, amsmath, amsthm}
\usepackage[bookmarksnumbered, colorlinks, plainpages]{hyperref}
\usepackage{mathrsfs}
\usepackage{cite}
\usepackage{tikz}
\usepackage{graphicx}
\usetikzlibrary{arrows.meta,calc,decorations.pathreplacing,positioning}
\usetikzlibrary{shapes.geometric}
\usepackage{float}
\usepackage{pgfplots}
\usetikzlibrary{shadows}

\usepgfplotslibrary{fillbetween}
\definecolor{curveblue}{RGB}{30,90,220}
\definecolor{negred}{RGB}{255,182,193}
\definecolor{posgreen}{RGB}{40,150,70}
\definecolor{rootdark}{RGB}{120,40,20}
\usepgfplotslibrary{fillbetween}
\definecolor{c1}{RGB}{70,40,200}

\definecolor{c2}{RGB}{213,14,0}

\definecolor{c3}{RGB}{0,158,10}

\definecolor{c4}{RGB}{230,126,34}
\newtheorem*{theoA}{Theorem A}
\newtheorem*{theoB}{Theorem B}
\newtheorem*{theoC}{Theorem C}
\newtheorem*{theoD}{Theorem D}
\newtheorem*{theoE}{Theorem E}

\newtheorem*{remA}{Remark A}
\newtheorem*{remB}{Remark B}
\newtheorem*{defiA}{Definition A}
\newtheorem*{defiB}{Definition B}
\newtheorem*{defiC}{Definition C}

\newtheorem*{cor A}{Corollary A}
\newtheorem*{cor B}{Corollary B}

\newtheorem{theo}{Theorem}[section]
\newtheorem{pro}{Proposition}[section]

\newtheorem{cor}{Corollary}[section]

\newtheorem{defi}{Definition}[section]
\newtheorem{rem}{Remark}[section]

\newcommand{\be}{\begin{equation}}
	\newcommand{\ee}{\end{equation}}
\newcommand{\beas}{\begin{eqnarray*}}
	\newcommand{\eeas}{\end{eqnarray*}}
\newcommand{\bea}{\begin{eqnarray}}
	\newcommand{\eea}{\end{eqnarray}}

\numberwithin{equation}{section}

\begin{document}
\title[A\MakeLowercase {lpha-convex functions in complex}......]{\LARGE A\Large\MakeLowercase {lpha-convex functions in complex} B\MakeLowercase{anach spaces}}

\date{}
\author[J. Sok\'o{\l}, S. Majumder, G. Haldar and S. Panja ]{Janusz Sok\'o{\l}$^{1}$,
Sujoy Majumder$^{2}$,  Goutam Haldar$^{3*}$  and Shantanu Panja$^4$}
\address{$^{1}$Faculty of Exact and Technical Sciences, University of Rzesz\'ow, Street: Prof. Pigonia 1, 35-310 Rzesz\'ow, Poland}
\email{jsokol@ur.edu.pl}
\address{$^{2}$Department of Mathematics, Raiganj University, Raiganj, West Bengal-733134, India.}
\email{sm05math@gmail.com, sjm@raiganjuniversity.ac.in}
\address{$^3$Department of Mathematics, Ghani Khan Choudhury Institute of Engineering and Technology, Narayanpur, Malda-732141, West Bengal, India.}
\email{goutamiitm@gmail.com, goutamiit1986@gmail.com}
\address{$^4$Department of Mathematics, University of Kalyani, West Bengal 741235, India.}
\email{panjasantu07@gmail.com}

\renewcommand{\thefootnote}{}
\footnote{2020 \emph{Mathematics Subject Classification}: 32H02, 30C45.}
\footnote{\emph{Key words and phrases}: Biholomorphic mapping, Starlike mapping, Quasi-convex mapping of type $B$ and order $\alpha$, Coefficient problems, Alpha-convex function, Hankel determinant.}
\footnote{*\emph{Corresponding Author}: Goutam Haldar.}
\renewcommand{\thefootnote}{\arabic{footnote}}
\setcounter{footnote}{0}
\begin{abstract}
Let $\mathbb{B}$ be the unit ball in a complex Banach space ${\mathbb X}$.  We study biholomorphic mappings $F: \mathbb{B} \to {\mathbb X}$ of the form $F(z) = f(z)z$ for $z \in \mathbb{B}$, where $f: \mathbb{B} \to \mathbb{C}$ is a holomorphic function satisfying $f(0) = 1$. Our main contributions are as follows: Inclusion Results: We establish inclusion conditions for $\mathcal{M}_{\alpha}(\mathbb{B})$, class of $\alpha$-convex functions in $\mathbb{B}$ to belong to the classes of starlike mappings or quasi-convex mappings of type $B$ on $\mathbb{B}$ for specified ranges of $\alpha$. Characterization: We prove that $F(z) = \frac{f(l_{u}(z))}{l_{u}(z)}z \in \mathcal{M}_{\alpha}(\mathbb{B})$ if and only if $f \in \mathcal{M}_{\alpha}(\mathbb{U})$, where $f \in \mathcal{S}(\mathbb{U})$, $u \in {\mathbb X}$ with $\|u\| = 1$, and $l_{u} \in T(u)$. Duality: We present a duality relationship between the classes $\mathcal{S}^*(\mathbb{U})$ and $\mathcal{M}_{\alpha}(\mathbb{B})$ for $\alpha \ge 0$. Hankel Determinant Bounds: We determine sharp upper bounds for the second Hankel determinant of mappings in $\mathcal{M}_\alpha(\mathbb{B})$ when $\alpha \ge 0$. The results established here directly extend the previous results to the broader setting of complex Banach spaces.
\end{abstract}
\thanks{Typeset by \AmS -\LaTeX}
\maketitle

\section{\bf Introduction}

Let $\mathcal{S}(\mathbb{U})$ denote the class of functions $f$ that are univalent in the open unit disk $\mathbb{U} = \{z \in \mathbb{C} : \vert{}z\vert{} < 1\}$, normalized by $f(0) = 0$ and $f'(0) = 1$. Any function $f \in \mathcal{S}(\mathbb{U})$ has a power series expansion of the form:
\begin{align}\label{eq1}
\displaystyle
f(z) = z + \sum_{n=2}^{\infty} a_n z^n, \quad z \in \mathbb{U}.
\end{align}

Let $\mathcal{P}$ denote the class of analytic functions $p$ in the open unit disk $\mathbb{U}$ satisfying $p(0) = 1$ and $\Re(p(z)) > 0$ for all $z \in \mathbb{U}$. Member of $\mathcal{P}$ is called a Carath\'eodory function. Coefficient bounds for functions in $\mathcal{P}$ play a fundamental role in solving various coefficient problems.\vspace{1.4mm}

A normalized univalent function $f \in \mathcal{S}(\mathbb{U})$ is said to be starlike (with respect to the origin) if and only if
$\Re\left\{ z f'(z)/f(z) \right\}> 0$, $z \in \mathbb{U}$.
The subclass consisting of all such starlike functions is denoted by $\mathcal{S}^*(\mathbb{U})$. Let $\mathcal{K}(\mathbb{U})$ denote the familiar subclass of $\mathcal{S}$ consisting of functions which are convex in $\mathbb{U}$. Then $f\in\mathcal{K}(\mathbb{U})$ if and only if
$\Re\left\{z f''(z)/f'(z) \right\} >-1$, $z \in \mathbb{U}$.
Next, we introduce another subclass of $\mathcal{S}(\mathbb{U})$, known as the class of $\alpha$-convex functions. This notion was introduced by Mocanu in 1969 \cite{Mocanu-1969} to construct a one-parameter family of subclasses of $\mathcal{S}(\mathbb{U})$ that provides a continuous passage from starlike to convex functions. More detailed treatments are available in Goodman \cite{Goodman-1983} and Mocanu et al.\cite{Mocanu-1999}. Related early concepts can also be traced back to Sakaguchi \cite{Sakaguchi-1962}. Additionally, $\alpha$-convex mappings in several complex variables have been investigated by Kohr \cite{Kohr-1998}.
\begin{defiA} Let $f\in\mathcal{S}(\mathbb{U})$ be such that $f(z)f'(z)/z \neq 0$, $z \in \mathbb U$. Then $f$ is said to be $\alpha$-convex in $\mathbb{U}$ if and only if
\begin{align}\label{Def1}
\Re\left\lbrace J_{\alpha}[f](z)\right\rbrace>0,\quad z\in\mathbb{U},
\end{align}
where
\begin{align*}
J_{\alpha}[f](z)=(1-\alpha)\frac{zf'(z)}{f(z)}+\alpha\left(1+\frac{zf''(z)}{f'(z)}\right).
\end{align*}
\end{defiA}
We denote the class of $\alpha$-convex functions by $\mathcal{M}_{\alpha}(\mathbb{U})$.

Certain authors refer to functions satisfying \eqref{Def1} as $\alpha$-starlike. It is straightforward to see that $\mathcal{M}_1(\mathbb{U}) = \mathcal{K}(\mathbb{U})$ and $\mathcal{M}_0(\mathbb{U})= \mathcal{S}^*(\mathbb{U})$; hence, $\mathcal{M}_\alpha(\mathbb{U})$ offers a continuous transition between convex and starlike functions. Furthermore, the parameter $\alpha$ is not required to lie within the interval $[0, 1]$.

\begin{remA} The assumption $f(z)f'(z)/z \neq 0$, $z \in \mathbb U$, in Definition A is redundant. In 1979, Sakaguchi and Fukui \emph{\cite{Sakaguchi-Fukui-1979}} established that $\Re\{J_{\alpha}[f](z)\}> 0$ for $z \in \mathbb U$, implies $f(z)f'(z)/z \neq 0$, $z \in \mathbb U$.	
\end{remA}

The following result of Miller et al. \cite{Miller-Mocanu-Reade-1973} shows that $\mathcal{M}_\alpha(\mathbb{U})$ is included in the class of starlike or convex functions for certain ranges of the parameter $\alpha$.
\begin{theoA}\label{Theorem A}\emph{\cite{Miller-Mocanu-Reade-1973}} If $f\in\mathcal{M}_{\alpha}(\mathbb{U})$, then it is starlike and univalent in $\mathbb{U}$. Moreover, if $\alpha\geq 1$, then $f$ is convex for $|z|<1$ and if $\alpha\leq -1$, then $1/f(1/z)$ is convex for $|z| > 1$.
\end{theoA}

Based on Theorem A, it can be made the following observation.
\begin{enumerate}
\item[1.] $\mathcal{M}_{\alpha}(\mathbb{U})\subseteq \mathcal{S}^*(\mathbb{U})$ for any real $\alpha$.
\item[2.] $\mathcal{M}_{\alpha}(\mathbb{U})\subseteq \mathcal{M}_{\beta}(\mathbb{U})\subseteq \mathcal{S}^*(\mathbb{U})$ for $0\leq \beta/\alpha\leq 1$.
\end{enumerate}

Todorov \cite{Todorov-1987} derived the following integral representation for $\alpha$-convex functions.
\begin{theoB}\emph{\cite{Todorov-1987}} The function $f\in \mathcal{M}_{\alpha}(\mathbb{U})$ for $\alpha>0$ if and only  if there exists a starlike function $F$ such that
\begin{align}\label{TheoB1}
f(z)=\left\lbrack\frac{1}{\alpha}\int\limits_{0}^{z}F^{1/\alpha}(\xi)\xi^{-1}{\rm d}\xi \right\rbrack^{\alpha}\quad z\in\mathbb{U},
\end{align}
where powers are meant as principal values.		
\end{theoB}

When investigating extremal problems in $\mathcal{M}_\alpha(\mathbb{U})$, the function corresponding to the Koebe function $k(z) = \frac{z}{(1-z)^2}$ serves as a natural candidate. Applying $k(z)$ to \eqref{TheoB1}, we get
\begin{align}\label{TheoB2}
K(\alpha,z):=\left\lbrack\frac{1}{\alpha}\int\limits_{0}^{z}\frac{\xi^{1/\alpha}}{\xi(1-\xi)^{2/\alpha}}{\rm d}\xi \right\rbrack^{\alpha}.
\end{align}

This function is called the $\alpha$-convex Koebe function.
For $\alpha>0$, it follows from \eqref{TheoB2} that
\begin{align}\label{TheoB3}
K(\alpha,r):=\left\lbrack\frac{1}{\alpha}\int\limits_{0}^{r}\rho^{1/\alpha-1}(1-\rho)^{-2/\alpha}{\rm d}\rho \right\rbrack^{\alpha}=r\left\lbrack {}_{2} F_{1}\left(\frac{1}{\alpha},\frac{2}{\alpha},\frac{1}{\alpha}+1;r\right)\right\rbrack^{\alpha},
\end{align}
where ${}_{2} F_{1}$ is the Gaussian hypergeometric function.\vspace{1.2mm}

Miller \cite{Miller-1973} established the following sharp upper bound for the second Taylor coefficient of functions in $\mathcal{M}_\alpha(\mathbb{U})$:
\begin{theoC}\emph{\cite{Miller-1973}} Let $f\in \mathcal{M}_\alpha(\mathbb{U})$ for $\alpha>0$ and be given by $f(z)=z+\sum_{n=2}^{\infty}a_n(\alpha)z^n$. Then $|a_2(\alpha)|\leq 2/(1+\alpha)$. The inequality is sharp.
\end{theoC}

Pommerenke \cite{Pommerenke-1966, Pommerenke-1967} defined the $s$-th Hankel determinant for $f \in \mathcal{S}(\mathbb{U})$ when $s \ge 1$ and $k \ge 1$ as
\begin{align*}
H_s(k) = \begin{vmatrix}
	a_k & a_{k+1} & \dots & a_{k+s-1} \\
	a_{k+1} & a_{k+2} & \dots & a_{k+s} \\
	\vdots & \vdots & \ddots & \vdots \\
	a_{k+s-1} & a_{k+s} & \dots & a_{k+2(s-1)}
\end{vmatrix},
\end{align*}
where $a_1 = 1$.\vspace{1.2mm}

Obtaining sharp bounds for Hankel determinants with coefficients from various subclasses of univalent functions has been an active area of recent research. However, most existing studies rely heavily on the methodology established in \cite{Janteng-Halim-Darus-2007}, where the sharp upper bounds of $H_2(2)$ were first derived for the classes $\mathcal{S}^*(\mathbb{U})$ and $\mathcal{K}(\mathbb{U})$.\vspace{1.2mm}

We first observe that several prior efforts to establish the sharp upper bound for $H_2(2)$ when $f \in \mathcal{M}_{\alpha}(\mathbb{U})$ yielded only non-sharp estimates (see, e.g.,\cite{Deekonda-Thoutreddy-2015,Raducanu-Zaprawa-2017,Shanmugan-Stephen-2012,Singh-Mehrok-2013}.\vspace{1.2mm}
In 2018, Sok\'{o}{\l} and Thomas \cite{Sokol-Thomas-2018} established further upper bound estimates for the second Hankel determinant of functions in $\mathcal{M}_\alpha(\mathbb{U})$ when $\alpha \ge 0$.
\begin{theoD}\emph{\cite[Theorem 1]{Sokol-Thomas-2018}} Let $f \in \mathcal{M}_\alpha(\mathbb{U})$, $\alpha \ge 0$, be of the form \eqref{eq1}. Then
\begin{align}\label{TheoD1}
|a_2 a_4 - a_3^2| \le
\begin{cases}
\dfrac{2(1 + 8\alpha + 3\alpha^2)}{(1 + 3\alpha)(2 + 15\alpha + 24\alpha^2 + 7\alpha^3)}, & \text{when } \alpha \in [0, 1], \\[10pt]
\dfrac{3 + 35\alpha + 126\alpha^2 + 162\alpha^3 + 79\alpha^4 + 27\alpha^5}{(1 + 2\alpha)^2 (1 + 3\alpha)(3 + 26\alpha + 42\alpha^2 + 18\alpha^3 + 7\alpha^4)}, & \text{when } \alpha \in [1, \infty).
\end{cases}
\end{align}
\end{theoD}

\noindent
Thus, in particular, if $f \in \mathcal{M}_0(\mathbb{U})=\mathcal{S}^*(\mathbb{U})$, then we have $H_2(2) \le 1$ and when $f \in \mathcal{M}_1(\mathbb{U})=\mathcal{K}(\mathbb{U})$, we have $H_2(2) \le 1/8$, the results obtained in \cite{Janteng-Halim-Darus-2007}. The first inequality in \eqref{TheoD1} was obtained earlier in \cite{Shanmugam-2012}.\vspace{1.2mm}

Also in the same paper, Sok\'{o}{\l} and Thomas \cite{Sokol-Thomas-2018} proved the following generalized version of Theorem D.
\begin{theoE}\emph{\cite[Theorem 2]{Sokol-Thomas-2018}} Let $f\in \mathcal{S}(\mathbb{U})$, $\alpha\in\mathbb{R}$ and $0\leq \beta\leq 1$. If
\begin{align}\label{TheoE1}
\Re\left\lbrace (1-\alpha)\frac{zf'(z)}{f(z)}+\alpha\left(1+\frac{zf''(z)}{f'(z)}\right)\right\rbrace>\beta,\quad z\in \mathbb{U},
\end{align}	
then
\begin{align}\label{TheoE2}
H_2(2) \le \begin{cases}
(1 - \beta)^2, & \text{if } \alpha = \beta = 0, \\[8pt]
\dfrac{(1 - \beta)^2 (S_1 + S_2)}{(\alpha + \beta)(1 + 3\alpha) S_3}, & \text{if } 0 < \alpha + 2\beta \le 1, \\[12pt]
\dfrac{(1 - \beta)^2 (S_4 + S_5)}{(1 + 2\alpha)^2 (1 + 3\alpha) S_6}, & \text{if } \alpha + 2\beta \ge 1,
\end{cases}
\end{align}
where
\begin{align*}
	S_1 &= 6\alpha^3 + 16\alpha^2 + 2\alpha, \\
	S_2 &= -6\alpha^2\beta + 2\beta - \beta^2 + 12\alpha\beta - 9\alpha\beta^2, \\
	S_3 &= 2 + 15\alpha + 24\alpha^2 + 7\alpha^3 - 10\alpha\beta - \beta - 13\alpha^2\beta, \\
	S_4 &= 3 + 35\alpha + 126\alpha^2 + 162\alpha^3 + 79\alpha^4 + 27\alpha^5, \\
	S_5 &= -4\beta - 52\alpha\beta - 184\alpha^2\beta - 180\alpha^3\beta - 12\alpha^4\beta + 2\beta^2 + 26\alpha\beta^2 + 92\alpha^2\beta^2 + 84\alpha^3\beta^2, \\
	S_6 &= 3 + 26\alpha + 42\alpha^2 + 18\alpha^3 + 7\alpha^4 - 4\beta - 40\alpha\beta - 52\alpha^2\beta + 2\beta^2 + 20\alpha\beta^2 + 26\alpha^2\beta^2.
\end{align*}
\end{theoE}

\begin{remB} For $\alpha = 0$, inequality \eqref{TheoE1} reduces to the class $\mathcal{S}^*_{\beta}(\mathbb{U})$ of starlike functions of order $\beta$, yielding $H_2(2) \le (1 - \beta)^2$, as established in \emph{\cite{Lee-2013}} \emph{(\text{see also} \cite{Cho-2018,Krishna-Ramreddy-2012})}.

\noindent For $\beta = 0$ the bounds in \eqref{TheoE2} reduce to \eqref{TheoD1} in Theorem D.
	
\noindent When $\alpha = 1$, \eqref{TheoE2} gives
\begin{align*}
H_2(2) \le (1 - \beta)^2 \frac{36 - 36\beta + 17\beta^2}{144(2 - 2\beta + \beta^2)},
\end{align*}
which represents a non-sharp bound for the class $\mathcal{K}_{\beta}(\mathbb{U})$ of convex functions of order $\beta$ obtained in \emph{\cite{Krishna-Ramreddy-2012}}. This bound yields the sharp value $1/8$ for $\alpha = 1$ and $\beta = 0$.
\end{remB}

Before proceeding further, let us define some recurring notation and terminology.

Let ${\mathbb X}$ be a complex Banach space with norm $\|\cdot\|$. For $r > 0$, the open ball centered at the origin with radius $r$ is defined as $\mathbb{B}_r = \{z \in {\mathbb X} : \|z\| < r\}$. For $r = 1$, we simply write $\mathbb{B}$. When ${\mathbb X} = \mathbb{C}$, the unit disk $\mathbb{B}$ is denoted by $\mathbb{U}$. The boundary and closure of $\mathbb{B}$ are denoted by $\partial \mathbb{B}$ and $\overline{\mathbb{B}}$, respectively.
\vspace{1.2mm}

Let $L({\mathbb X}, {\mathbb Y})$ denote the space of continuous linear operators from a complex Banach space ${\mathbb X}$ into a complex Banach space ${\mathbb Y}$, and let $I$ denote the identity operator on ${\mathbb X}$. For each $z \in {\mathbb X} \setminus \{0\}$, define the set
\begin{align*}
\displaystyle
T(z) = \{ l_z \in L({\mathbb X}, \mathbb{C}) : l_z(z) = \|z\|, \quad \|l_z\| = 1 \},
\end{align*}
where $\|l_z\|$ is the supremum norm of an continuous linear operator $A\in L({\mathbb X}, {\mathbb Y})$:
\begin{align*}
\|A\|=\sup\{\|A(z)\|: \|z\|=1\}.
\end{align*}
By the Hahn–-Banach theorem, $T(z)$ is non-empty for every non-zero $z \in {\mathbb X}$.\vspace{1.2mm}

Let $\mathcal{H}(\Omega, \Omega')$ represent the set of all holomorphic mappings from a domain $\Omega \subset {\mathbb X}$ into a domain $\Omega' \subset {\mathbb Y}$. We abbreviate $\mathcal{H}(\Omega, {\mathbb X})$ as $\mathcal{H}(\Omega)$. If $f \in \mathcal{H}(\mathbb{B})$, then for all $w$ in a sufficiently small neighborhood of $z \in \mathbb{B}$, $f$ admits the power series expansion\begin{align*}f(w) = \sum_{n=0}^{\infty} \frac{1}{n!} D^n f(z)(w - z)^n,
\end{align*}
where $D^n f(z)$ denotes the $n$-th Fr\'{e}chet derivative of $f$ at $z$, represented by a bounded symmetric $n$-linear mapping from ${\mathbb X}^n$ into ${\mathbb X}$. Here, $D^n f(z)(w - z)^n$ is standard notation for evaluating the $n$-linear form at $n$ identical copies of $(w - z)$:
\begin{align*}
\displaystyle
D^n f(z)(w - z)^n = D^n f(z)(\underbrace{w - z, w - z, \dots, w - z}_{n \text{ times}}),
\end{align*}
with the convention $D^0 f(z)(z) = f(z)$.\vspace{1.2mm}

Let $\mathbb{B}$ be the open unit ball in a complex Banach space ${\mathbb X}$. A mapping $f \in \mathcal{H}(\mathbb{B})$ is said to be normalized if $f(0) = 0$ and $Df(0) = I$, where $I$ is the identity operator on ${\mathbb X}$.
We say that $f$ is:\\
$\bullet$ biholomorphic if it is a bijection onto its image $f(\mathbb{B})$ with a holomorphic inverse $f^{-1}: f(\mathbb{B}) \to \mathbb{B}$;\\
$\bullet$ locally biholomorphic if its Fr\'{e}chet derivative $Df(z)$ has a bounded inverse for every $z \in \mathbb{B}$.

The family of all normalized biholomorphic mappings on $\mathbb{B}$ is denoted by $\mathcal{S}(\mathbb{B})$.\vspace{1.6mm}

 A biholomorphic mapping on $\mathbb{B}$ is said to be starlike if $f(\mathbb{B})$ is starlike with
 respect to the origin. If $\mathbb{B}$ is the unit ball of ${\mathbb X}$, a complex Banach space, we let $\mathcal{S}^*(\mathbb{B})$ denote the subset of $\mathcal{S}(\mathbb{B})$ consisting of normalized starlike mappings from $\mathbb{B}$ into ${\mathbb X}$. We say that $f:\mathbb X\rightarrow\mathbb Y$ is
 \emph{convex} on $\Omega$ if $f$ is biholomorphic on $\Omega\subset\mathbb X$ and
 $f(\Omega)$ is a convex domain in ${\mathbb Y}$.\vspace{1.6mm}

In several complex variables, classical one-dimensional results often fail to hold in their original form. Most notably, Poincar\'{e} disproved the higher-dimensional analogue of the Riemann mapping theorem, while Cartan showed that the Bieberbach conjecture does not hold on the unit ball. Consequently, extending univalent function theory to higher dimensions necessitates the introduction of suitable sub-classes or additional structural constraints.

Recently, coefficient bounds for various sub-classes of normalized biholomorphic mappings in higher dimensions have been established by several authors; see, for example, Bracci \cite{Bracci-2015}, Bracci et al. \cite{Bracci-2016}, Bracci and Roth \cite{Bracci-2018}, Graham et al. \cite{Graham-2002, Graham-2016, Graham-2017}, Hamada and Honda \cite{Hamada-Honda-2008}, Hamada et al. \cite{Hamada-Honda-Kohr-2006} and Hamada and Kohr \cite{Hamada-Kohr-2020}.\vspace{1.6mm}

In this section, we recall the definition of starlikeness of order $\beta$ in complex Banach space.
\begin{defiB}\emph{\cite[Definition 1.2]{Hamada-2001}}
Let $f:\mathbb{B}\to {\mathbb X}$ be a normalized locally biholomorphic mapping and let $\beta\in [0,1)$. We say that $f$ is starlike of order $\beta$ if
\begin{align*}
\displaystyle
\Re\left\lbrace \frac{||z||}{l_z([Df(z)]^{-1}f(z))}\right\rbrace>\beta,\quad z\in\mathbb{B}\backslash\{0\},\quad l_z\in T(z).
\end{align*}
\end{defiB}
Because $f$ is locally biholomorphic then its Fr\'{e}chet derivative $Df(z)$ has a bounded inverse $[Df(z)]^{-1}f(z)$ for every $z \in \mathbb{B}$
Let $\mathcal{S}^*_{\beta}(\mathbb{B})$ denote the class of starlike mappings of order $\beta$ on $\mathbb{B}$. When $\beta=0$, the class $\mathcal{S}^*_{\beta}(\mathbb{B})$ reduces to the standard starlike class $\mathcal{S}^*(\mathbb{B})$.
\vspace{1.2mm}

There are now many references for starlike mappings in higher dimensions. For more information, see the recent monographs by Gong  \cite{Gong} and Kohr \& Liczberski  \cite{Kohr- Liczberski-1998}. Some of the foundational and original papers in this area include
\cite{Bomba,Dlugosz,Chu,SSQ_CHMS_1990,Su2,Thomas-Tuneski,Xu-He-Xu,Su3,Trybucka}.

\begin{defiC}\cite[Definition 1,1]{Liu-Liu-2010}
A normalized locally biholomorphic mapping $f\in\mathcal{H}(\mathbb{B})$ is said to be quasi-convex of type $B$ and order $\beta$ if
\begin{align*}
\displaystyle
\Re\left\lbrace \frac{l_z\left([Df(z)]^{-1}\left(D^2f(z)(z,z)+Df(z)z\right)\right)}{||z||}\right\rbrace>\beta,\quad z\in\mathbb{B}\backslash\{0\},\quad l_z\in T(z).
\end{align*}
\end{defiC}

If $\mathbb{B}$ is the unit ball of ${\mathbb X}$, we let $\mathcal{K}^{\beta}_B(\mathbb{B})$ denote the subset of $\mathcal{S}(\mathbb{B})$
consisting of normalized quasi-convex mappings of type $B$ and order $\beta$ from $\mathbb{B}$ into ${\mathbb X}$. When $\beta=0$, the class $\mathcal{K}^{\beta}_B(\mathbb{B})$ reduces to  $\mathcal{K}_B(\mathbb{B})$, the class of quasi-convex mappings of type $B$.\vspace{1.2mm}
\par Convex mappings in higher dimensions have been extensively studied in recent literature. For foundational work on this topic, see \cite{Ham4,Rop-Su,Su2,Su3,Xu-Dl-Li,Xu-Yang-Liu-Xu-2015,Zhang-Liu-2002}.\vspace{1.6mm}
\par Next, we define $\alpha$-convex function in complex Banach space.
\begin{defi}
Let $f:\mathbb{B}\to {\mathbb X}$ be a normalized locally biholomorphic mapping and let $\alpha\in \mathbb{R}$. Then $f$ is said to be $\alpha$-convex in $\mathbb{B}$ if and only if
\begin{align*}
\displaystyle
\Re\left\lbrace (1-\alpha)\frac{||z||}{l_z\left([DF(z)]^{-1}F(z)\right)} +\alpha\frac{l_z\left([DF(z)]^{-1}(D^2 F(z)(z^2)+DF(z) z)\right)}{||z||}\right\rbrace>0,
\end{align*}
$z\in\mathbb{B}\backslash\{0\}$ and $l_z\in T(z)$.	
\end{defi}

Let $\mathcal{M}_{\alpha}(\mathbb{B})$ denote the class of $\alpha$-convex in $\mathbb{B}$. It is straightforward to see that $\mathcal{M}_1(\mathbb{B}) = \mathcal{K}_B(\mathbb{B})$ and $\mathcal{M}_0(\mathbb{B})= \mathcal{S}^*(\mathbb{B})$.\vspace{2mm}

In this paper, by considering biholomorphic mappings $F:\mathbb{B}\to {\mathbb X}$ of the form $F(z)=f(z)z$ for $z\in\mathbb{B}$, where $f:\mathbb{B}\to \mathbb{C}$ is a holomorphic function satisfying $f(0)=1$, we establish:\vspace{1.2mm}

$\bullet$ inclusive conditions for $\mathcal{M}_{\alpha}(\mathbb{B})$ to belong to the classes of starlike mappings or quasi-convex mappings of type $B$ in $\mathbb{B}$ for specified ranges of $\alpha$;\vspace{1.2mm}

$\bullet$ a characterization showing that $F(z)=\frac{f(l_{u}(z))}{l_{u}(z)}z \in \mathcal{M}_{\alpha}(\mathbb{B})$ if and only if $f \in \mathcal{M}_{\alpha}(\mathbb{U})$, where $f \in \mathcal{S}(\mathbb{U})$, $u \in {\mathbb X}$ with $\|u\| = 1$, and $l_{u} \in T(u)$;\vspace{1.2mm}

$\bullet$ a duality relationship between the classes $\mathcal{S}^*(\mathbb{U})$ and $\mathcal{M}_{\alpha}(\mathbb{B})$ for $\alpha\geq 0$;\vspace{1.2mm}

$\bullet$ sharp upper bound estimates for the second Hankel determinant of mappings in $\mathcal{M}_\alpha(\mathbb{B})$ when $\alpha \ge 0$.

The results obtained here directly extend Theorems A–-E to the setting of complex Banach spaces.

\section{{\bf Some of the basic results about $\alpha$-convex function }}\label{Sec-2}

By considering biholomorphic mappings $F:\mathbb{B}\to {\mathbb X}$ of the form $F(z)=f(z)z$ ($z\in\mathbb{B}$), where $f:\mathbb{B}\to \mathbb{C}$ is a holomorphic function with $f(0)=1$, we extend Theorems A to complex Banach spaces as follows.

 \begin{theo}\label{Th-1}Let $\mathbb{B}$ be the unit ball in a complex Banach space ${\mathbb X}$. Let $f\in\mathcal{H}(\mathbb{B},\mathbb{C})$ with $f(0)=1$ and $F(z)=zf(z)\in\mathcal{M}_{\alpha}(\mathbb{B})$. Then $F\in\mathcal{S}^*(\mathbb{B})$. If $\alpha\geq 1$, then $F\in\mathcal{K}_B(\mathbb{B})$. Moreover if $F(z)=zf(z)\in\mathcal{M}_{\beta}(\mathbb{B})$, then $F(z)=zf(z)\in\mathcal{M}_{\alpha}(\mathbb{B})$ for all $\alpha,
 \beta\in\mathbb{R}$, $0\leq \alpha/\beta\leq 1$.
 \end{theo}

\begin{proof} By the given condition, we have $F\in\mathcal{M}_{\alpha}(\mathbb{B})$ and so
\begin{align}\label{Lem1.1}
\displaystyle
\Re\left\lbrace (1-\alpha)\frac{||z||}{l_z\left([DF(z)]^{-1}F(z)\right)} +\alpha\frac{l_z\left([DF(z)]^{-1}\phi(z)\right)}{||z||}\right\rbrace>0,
\end{align}
for $z\in\mathbb{B}\backslash\{0\}$ and $l_z\in T(z)$, where $\phi(z)=D^2 F(z)(z^2)+DF(z) z$.	
	
Let $z\in \mathbb{B}\backslash\{0\}$ be fixed and let $z_0=z/\|z\|$.
We define the function $H:\mathbb{U}\to \mathbb{C}$ as follows
\begin{align}\label{Lem1.2}	
\displaystyle
H(\zeta)=\begin{cases}
(1-\alpha)\frac{\zeta}{l_{z}\left([DF(\zeta z_0)]^{-1}F(\zeta z_0)\right)}+\alpha \frac{l_{z}\left([DF(\zeta z_0)]^{-1}\phi(\zeta z_0)\right)}{\zeta},& \zeta\neq 0,\\[8pt]
1& \zeta=0.
\end{cases}
\end{align}

\noindent Clearly $H\in\mathcal{H}(\mathbb{U})$ and $H(0)=1$.
Since $l_{\zeta z}=\frac{|\zeta|}{\zeta}l_z$ for each $\zeta\in\mathbb{U}\backslash\{0\}$, we have
\begin{align}\label{Lem1.3}
   H(\zeta)
   =(1-\alpha)\frac{||\zeta z_0||}{l_{\zeta z_0}\left([DF(\zeta z_0)]^{-1}F(\zeta z_0)\right)}+\alpha \frac{l_{\zeta z_0}\left([DF(\zeta z_0)]^{-1}\phi(\zeta z_0)\right)}{||\zeta z_0||}
\end{align}	
and so $\Re\{H(\zeta)\}>0$ for all $\zeta\in\mathbb{U}$.
Using a similar method as in Theorem 7.1.14 \cite{Graham-Kohr}, we have
\begin{align}\label{Lem1.4}
\displaystyle
[D F(z)]^{-1}F(z)=\frac{f(z) z}{f(z)+D f(z)z},
\end{align}
where $D f(z):\mathbb B\rightarrow\mathbb C$ because $f\in\mathcal{H}(\mathbb{B},\mathbb{C})$, $D F(z):\mathbb B\rightarrow\mathbb {\mathbb X}$, and
\begin{align}\label{Lem1.5}
\displaystyle
[D F(z)]^{-1}\phi(z)=\frac{\phi(z)}{f(z)+D f(z)z},
\end{align}
where $\phi(z)=D^2 F(z)(z^2)+DF(z) z$. Now from \eqref{Lem1.4}, it is clear that
\begin{align}\label{Lem1.6}
\frac{||\zeta z_0||}{l_{\zeta z_0}\left([DF(\zeta z_0)]^{-1}F(\zeta z_0)\right)}=\frac{f(\zeta z_0)+D f(\zeta z_0)\zeta z_0}{f(\zeta z_0)}.
\end{align}

\noindent Suppose $g(\zeta)=f(\zeta z_0)\zeta$. Then we have
\begin{align}\label{Lem1.7}
g'(\zeta)=f(\zeta z_0)+Df(\zeta z_0)\zeta z_0.
\end{align}

\noindent Since $g(\zeta)=f(\zeta z_0)\zeta$, using \eqref{Lem1.7} to \eqref{Lem1.6}, we get
\begin{align}\label{Lem1.8}
\frac{||\zeta z_0||}{l_{\zeta z_0}\left([DF(\zeta z_0)]^{-1}F(\zeta z_0)\right)}=\zeta\frac{g'(\zeta)}{g(\zeta)}.
\end{align}

\noindent Now we will be trying to express the  part with the function $\phi$ in \eqref{Lem1.3} in terms of the function $g$. As the function $F$ is normalized, it has the power series expansion
\begin{equation}\label{Lem1.9a}
F(z)=z+\sum\limits_{k=2}^{\infty}\frac{1}{k!}D^k F(0)(z^k).
\end{equation}
It is known that
\begin{equation*}
  \frac{1}{k!}D^kF(z)(h^k)=\frac{1}{2\pi i}\int_{|\zeta|=r}\frac{F(z+\zeta h)}{\zeta^{k+1}}{\rm d}\zeta,
\end{equation*}
so
\begin{align*}
  \phi(z)=D^2 F(z)(z^2)+DF(z)(z)=&\frac{1}{2\pi i}\int_{|\zeta|=r}\frac{(2+\zeta)F((1+\zeta)z)}{\zeta^{3}}{\rm d}\zeta.
\end{align*}
\noindent Using \eqref{Lem1.9a} in this integral, we get
\begin{align}\label{Lem1.11}
\displaystyle
\phi(z)=D^2 F(z)(z^2)+DF(z)
(z)=&z+\sum\limits_{k=2}^{\infty}\frac{k^2}{k!}D^k F(0)(z^k).
\end{align}

\noindent Since $F(z)=zf(z)$, we have (see \cite[Theorem 2]{Xiong-Yu-Wang-2025})
\begin{align}\label{Th-1x}
\displaystyle
D^mF(0)(z^m)=mD^{m-1}f(0)(z^{m-1})z,
\end{align}
where $m\in\mathbb{N}$. As $z_0=z/||z||$, $l_{z}=\|z\|$, it follows from \eqref{Th-1x} that
\begin{align}\label{Th-1y}
\displaystyle
\frac{l_{z}(D^mF(0)(z^m))}{m!||z||^m}=\frac{D^{m-1}f(0)(z_0^{m-1})}{(m-1)!}.
\end{align}

\noindent Now substituting \eqref{Th-1x} into \eqref{Lem1.9a} and \eqref{Lem1.11}, we get respectively
\begin{equation}\label{Lem1.12}
\displaystyle
F(z)=z+\sum\limits_{k=2}^{\infty}\frac{k}{k!}D^{k-1} f(0)(z^{k-1}) z
\end{equation}
and
\begin{equation}\label{Lem1.13}
\displaystyle
\phi(z)=D^2 F(z)(z^2)+DF(z)
(z)=z+\sum\limits_{k=2}^{\infty}\frac{k^3}{k!}D^{k-1} f(0)(z^{k-1}) z.
\end{equation}

\noindent Therefore, we have
\begin{align}\label{Lem1.13a}
\displaystyle
\phi(\zeta z_0)=\zeta z_0\left(1+\sum\limits_{m=1}^{\infty}\frac{(m+1)^2}{m!}D^{m} f(0)(z_0^{m})\zeta^m\right).
\end{align}

\noindent If
\begin{align}\label{Lem1.14}
\displaystyle
f(\zeta z_0)=1+\sum\limits_{k=1}^{\infty}\frac{1}{k!}D^k f(0)(z_0^k)\zeta^k,
\end{align}
then
\begin{align}\label{Lem1.15}
\displaystyle
D f(\zeta z_0)(\zeta z_0)=\sum\limits_{k=1}^{\infty}\frac{k}{k!}D^k f(0)(z_0^k)\zeta^k.
\end{align}

\noindent In view of \eqref{Lem1.7} and adding \eqref{Lem1.14} and \eqref{Lem1.15}, we obtain
\begin{align}\label{Lem1.16}
\displaystyle
g'(\zeta)=f(\zeta z_0)+D f(\zeta z_0)\zeta z_0=1+\sum\limits_{k=1}^{\infty}\frac{k+1}{k!}D^k f(0)(z_0^k)\zeta^k.
\end{align}	

\noindent Clearly
\begin{align}\label{Lem1.17}
\displaystyle
\zeta g''(\zeta)+g'(\zeta)=1+\sum\limits_{k=1}^{\infty}\frac{(k+1)^2}{k!}D^k f(0)(z_0^k)\zeta^k.
\end{align}

\noindent Now from \eqref{Lem1.5}, \eqref{Lem1.13a} and \eqref{Lem1.16}, we get
\begin{align}\label{Lem1.18}
\displaystyle
\frac{l_{\zeta z_0}\left([D F(\zeta z_0)]^{-1}\phi(\zeta z_0)\right)}{||\zeta z_0||}=\frac{1+\sum\limits_{m=1}^{\infty}\frac{(m+1)^2}{m!}D^{m} f(0)(z^{m})\zeta^m}{1+\sum\limits_{k=1}^{\infty}\frac{k+1}{k!}D^k f(0)(z_0^k)\zeta^k}.
\end{align}

\noindent Using \eqref{Lem1.16} and \eqref{Lem1.17} to \eqref{Lem1.18}, we express the  second part of \eqref{Lem1.3} in terms of the function $g$
\begin{align}\label{Lem1.19}
\frac{l_{\zeta z_0}\left([D F(\zeta z_0)]^{-1}\phi(\zeta z_0)\right)}{||\zeta z_0||}
	=&\frac{\zeta g''(\zeta)+g'(\zeta)}{g'(\zeta)}=1+\zeta\frac{
	g''(\zeta)}{g'(\zeta)}.
\end{align}

\noindent Now in view of \eqref{Lem1.8} and  \eqref{Lem1.19}, we find that
\begin{align*}\label{Lem1.21}
\Re\{H(\zeta)\}=&\Re\left\{(1-\alpha)\frac{||\zeta z_0||}{l_{\zeta z_0}\left([DF(\zeta z_0)]^{-1}F(\zeta z_0)\right)}+\alpha \frac{l_{\zeta z_0}\left([DF(\zeta z_0)]^{-1}\phi(\zeta z_0)\right)}{||\zeta z_0||}\right\}\\
        =&\Re\left\{(1-\alpha)\frac{\zeta g'(\zeta)}{g(\zeta)}+\alpha\left\{1+\frac{
	\zeta g''(\zeta)}{g'(\zeta)}\right\}\right\}>0, \quad (\zeta\in\mathbb U).
\end{align*}
Therefore, the function $g$ is $\alpha$-convex function defined in Definition A and so, from Theorem A, we have
\begin{equation*}
  \Re\left\{\frac{\zeta g'(\zeta)}{g(\zeta)}\right\}=\Re\left\{\frac{||\zeta z_0||}{l_{\zeta z_0}\left([DF(\zeta z_0)]^{-1}F(\zeta z_0)\right)}\right\}>0, \quad (\zeta\in\mathbb U)
\end{equation*}
so the equivalence between \eqref{Lem1.2} and \eqref{Lem1.3} implies that $F\in\mathcal{S}^*(\mathbb{B})$. Furthermore, from Theorem A, we have for $\alpha\geq 1$
\begin{equation*}
  \Re\left\{1+\frac{\zeta g''(\zeta)}{g'(\zeta)}\right\}=\Re\left\{\frac{l_{\zeta z_0}\left([DF(\zeta z_0)]^{-1}\phi(\zeta z_0)\right)}{||\zeta z_0||}\right\}>0, \quad (\zeta\in\mathbb U)
\end{equation*}
So, if $\alpha\geq 1$, then $F\in\mathcal{K}_B(\mathbb{B})$. Furthermore, the observations after Theorem A imply that if $F(z)=zf(z)\in\mathcal{M}_{\beta}(\mathbb{B})$, then $F(z)=zf(z)\in\mathcal{M}_{\alpha}(\mathbb{B})$ for all $\alpha,
 \beta\in\mathbb{R}$, $0\leq \alpha/\beta\leq 1$.

\end{proof}

 To establish our main results, we first introduce the following Propositions. In particular, Proposition \ref{pro-1} provides the extremal mappings needed for our main theorems.
\begin{pro}\label{pro-1} Let $f \in \mathcal{S}(\mathbb{U})$, $u \in {\mathbb X}$ with $\|u\| = 1$ and $l_{u} \in T(u)$. Then $F$ defined by
$F(z)=\frac{f(l_{u}(z))}{l_{u}(z)}z$ belongs to $\mathcal{M}_{\alpha}(\mathbb{B})$ if and only if $f \in \mathcal{M}_{\alpha}(\mathbb{U})$.
\end{pro}
\begin{proof} Let $g(z)=\frac{f(l_{u}(z))}{l_{u}(z)}$. Since $f\in\mathcal{S}(\mathbb{U})$, it follows that
\begin{align*}
\frac{f'(l_{u}(z))l_{u}(z)}{f(l_{u}(z))}\neq 0, \quad z\in\mathbb{B}.
\end{align*}
	
\noindent Straightforward calculation yields
\begin{align}\label{Lem2.1}
\frac{Dg(z) z}{g(z)}+1=\frac{f'(l_{u}(z))l_{u}(z)}{f(l_{u}(z))}.
\end{align}	
	
\noindent Clearly $Dg(z) z/g(z)+1\neq 0$ for all $z\in\mathbb{B}$. Using a similar method as in Theorem 7.1.14 \cite{Graham-Kohr}, we have
\begin{align*}
\displaystyle
[D F(z)]^{-1}\eta=\frac{1}{g(z)}\left\lbrack \eta-\frac{(Dg(z)\eta)z}{g(z)+D g(z)z}\right\rbrack,\quad \eta\in {\mathbb X}.
\end{align*}

\noindent Hence, $F$ is a normalized locally biholomorphic on $\mathbb{B}$, and
\begin{align*}
\displaystyle
[D F(z)]^{-1}F(z)=\frac{f(l_{u}(z))}{f'(l_{u}(z))l_{u}(z)}z
\end{align*}
and so
\begin{align}\label{Lem2.2}
\frac{||z||}{l_z\left([D F(z)]^{-1}F(z)\right)}=\frac{f'(l_{u}(z))l_{u}(z)}{f(l_{u}(z))}.
\end{align}

\noindent Since $F(z)=g(z)z$, it follows that $D F(z)\eta=(Dg(z)\eta)z+g(z)\eta$, $\eta\in {\mathbb X}$. Consequently
$D F(z)z=f'(l_u(z))z$. Therefore
\begin{align*}
D^2 F(z)(z^2)+DF(z)z=D(DF(z)z)z=f''(l_u(z))l_u(z)z+f'(l_u(z))z,
\end{align*}
which shows that $D^2 F(z)(z^2)=f''(l_u(z))l_u(z)z$.
Note that
\begin{align*}
&[DF(z)]^{-1}(D^2 F(z)(z^2)+DF(z)z)\\=&z+[DF(z)]^{-1}(D^2 F(z)(z^2))\\=&
z+\frac{1}{g(z)}\left\lbrack f''(l_u(z))l_u(z)z-\frac{Dg(z)z}{g(z)+Dg(z)z}f''(l_u(z))l_u(z)z\right\rbrack, \quad z\in\mathbb{B}
\end{align*}
and so from \eqref{Lem2.1}, we obtain
\begin{align*}
[DF(z)]^{-1}(D^2 F(z)(z^2)+DF(z)z)=\left\lbrack 1+\frac{f''(l_u(z))l_u(z)}{f'(l_u(z))}\right\rbrack z,\quad z\in\mathbb{B}.
\end{align*}

\noindent Clearly
\begin{align}\label{Lem2.3}
\frac{l_z\left([DF(z)]^{-1}(D^2 F(z)(z^2)+DF(z)z)\right)}{||z||}=1+\frac{f''(l_u(z))l_u(z)}{f'(l_u(z))},\quad z\in\mathbb{B}.
\end{align}

\noindent Now from \eqref{Lem2.2} and \eqref{Lem2.3}, we see that
\begin{align*}
&\Re\left\lbrace (1-\alpha)\frac{||z||}{l_z\left([DF(z)]^{-1}F(z)\right)} +\alpha\frac{l_z\left([DF(z)]^{-1}(D^2 F(z)(z^2)+DF(z)z)\right)}{||z||}\right\rbrace\\=&\Re\left\lbrace (1-\alpha)\frac{f'(l_{u}(z))l_{u}(z)}{f(l_{u}(z))}+\alpha\left(1+\frac{f''(l_u(z))l_u(z)}{f'(l_u(z))}\right)\right\rbrace>0
\end{align*}	
if and only if
\begin{align*}
\Re\left\lbrace (1-\alpha)\frac{\xi f'(\xi)}{f(\xi)}+\alpha\left(1+\frac{\xi f''(\xi)}{f'(\xi)}\right)\right\rbrace>0, \quad \xi\in\mathbb{U}.
\end{align*}
This shows that
$F(z)=\frac{f(l_{u}(z))}{l_{u}(z)}z$ belongs to $\mathcal{M}_{\alpha}(\mathbb{B})$ if and only if $f \in \mathcal{M}_{\alpha}(\mathbb{U})$.
\end{proof}
\begin{pro}\label{pro-2} Let $f \in \mathcal{H}(\mathbb{B},\mathbb{C})$ with $f(0)=1$ and $F(z)=f(z)z$. Let $g(\zeta)=f(\zeta z_0)\zeta$, where $\zeta\in\mathbb{U}$ and $z_0\in\partial \mathbb{B}$ is fixed. Then $g\in\mathcal{M}_{\alpha}(\mathbb{U})$ if and only if $F\in\mathcal{M}_{\alpha}(\mathbb{B})$.
\end{pro}
\begin{proof}
First we suppose that $F\in\mathcal{M}_{\alpha}(\mathbb{B})$. Clearly $F$ is biholomorphic on $\mathbb{B}$. Since $F(z)=f(z)z$, it follows that $D F(z)z=(Df(z)z)z+f(z)z$, $z\in \mathbb{B}$. Thus $f(z)\neq 0$ and $f(z)+Df(z)z\neq 0$ for all $z\in\mathbb{B}$. Since $g(\zeta)=f(\zeta z_0)\zeta$, we using \eqref{Lem1.8}, we get
\begin{align}\label{Lem3.1}
\frac{||\zeta z_0||}{l_{\zeta z_0}\left([DF(\zeta z_0)]^{-1}F(\zeta z_0)\right)}=\zeta\frac{g'(\zeta)}{g(\zeta)}.
\end{align}

\noindent Let $\phi(z)=D^2 F(z)(z^2)+DF(z)z)$.
On the other hand, \eqref{Lem1.19} gives
\begin{align}\label{Lem3.2}
\displaystyle
\frac{l_{\zeta z_0}\left([D F(\zeta z_0)]^{-1}\phi(\zeta z_0)\right)}{||\zeta z_0||}
=1+\zeta\frac{
g''(\zeta)}{g'(\zeta)}.
\end{align}

\noindent Therefore, using \eqref{Lem3.1} and \eqref{Lem3.2}, we obtain that
\begin{align*}
&\Re\left\lbrace (1-\alpha)\frac{\zeta g'(\zeta)}{g(\zeta)}+\alpha\left(1+\frac{\zeta g''(\zeta)}{g'(\zeta)}\right)\right\rbrace\\=
&\Re\left\lbrace (1-\alpha)\frac{||\zeta z_0||}{l_{\zeta z_0}\left([DF(\zeta z_0)]^{-1}F(\zeta z_0)\right)} +\alpha\frac{l_{\zeta z_0}\left([DF(\zeta z_0)]^{-1}\phi(\zeta z_0)\right)}{||\zeta z_0||}\right\rbrace>0,
\end{align*}
which implies that $g\in\mathcal{M}_{\alpha}(\mathbb{U})$.
\vspace{1.2mm}

\noindent Conversely, we assume $g\in\mathcal{M}_{\alpha}(\mathbb{U})$. Clearly $\frac{\zeta g'(\zeta)}{g(\zeta)}\neq 0$, $\zeta\in\mathbb{U}$. Therefore from \eqref{Lem2.1}, we have
\begin{align*}
\frac{Df(z) z}{f(z)}+1\neq 0, \quad z\in\mathbb{B}.
\end{align*}	

\noindent From this and \eqref{Lem1.4}, we readily deduce that $F$ is locally biholomorphic on $\mathbb{B}$. Let us take $z=\zeta z_0$, where $\zeta=||z||$.
Then from \eqref{Lem3.1} and \eqref{Lem3.2}, we get
\begin{align*}
&\Re\left\lbrace (1-\alpha)\frac{||z||}{l_{z}\left([DF(z)]^{-1}F(z)\right)} +\alpha\frac{l_{z}\left([DF(z)]^{-1}\phi(z)\right)}{||z||}\right\rbrace\\&=
\Re\left\lbrace (1-\alpha)\frac{\zeta g'(\zeta)}{g(\zeta)}+\alpha\left(1+\frac{\zeta g''(\zeta)}{g'(\zeta)}\right)\right\rbrace>0,
\end{align*}
which implies that $F\in\mathcal{M}_{\alpha}(\mathbb{B})$.
\end{proof}

\begin{theo}\label{Th-2} Let $f \in \mathcal{S}(\mathbb{U})$, $u \in {\mathbb X}$ with $\|u\| = 1$ and $l_{u} \in T(u)$. Then
$F(z)=\frac{f(l_{u}(z))}{l_{u}(z)}z\in\mathcal{M}_{\alpha}(\mathbb{B})$ for $\alpha>0$ if and only if there exists $\tilde F(z)=\frac{\tilde f(l_{u}(z))}{l_{u}(z)}z\in\mathcal{S}^*(\mathbb{B})$, where $\tilde f \in \mathcal{S}(\mathbb{U})$ such that
\begin{align}\label{Th-2.0}
f(l_u(z))=\left\lbrack \frac{1}{\alpha} \int_{0}^{l_u(z)} \tilde f^{1/\alpha}(t) t^{-1}\, {\rm d}t \right\rbrack^{\alpha},
\end{align}
where powers are meant as principal values.
\end{theo}

\begin{proof}
First we suppose $F(z)=\frac{f(l_{u}(z))}{l_{u}(z)}z\in\mathcal{M}_{\alpha}(\mathbb{B})$. Then by Proposition \ref{pro-1}, we have $f\in\mathcal{M}_{\alpha}(\mathbb{U})$.
Now from \eqref{Lem2.2} and \eqref{Lem2.3}, we see that
\begin{align}\label{Th-2.1}
&(1-\alpha)\frac{||z||}{l_z\left([DF(z)]^{-1}F(z)\right)} +\alpha\frac{l_z\left([DF(z)]^{-1}(D^2 F(z)(z^2)+DF(z)z)\right)}{||z||}\\=& (1-\alpha)\frac{f'(l_{u}(z))l_{u}(z)}{f(l_{u}(z))}+\alpha\left(1+\frac{f''(l_u(z))l_u(z)}{f'(l_u(z))}\right)=p(l_{u}(z)),\nonumber
\end{align}	
where $p\in\mathcal{P}$.

\noindent Suppose the function $\tilde F(z)=\frac{\tilde f(l_{u}(z))}{l_{u}(z)}z$ is defined by
\begin{align}\label{Th-2.2}
\frac{\tilde f'(l_{u}(z))l_{u}(z)}{\tilde f(l_{u}(z))}=p(l_u(z)).	
\end{align}

\noindent Clearly $\tilde f\in\mathcal{S}^*(\mathbb{U})$. According to Lemma 2.1 \cite{Xu-He-Xu}, we can say that $\tilde F\in\mathcal{S}^*(\mathbb{B})$.
 Let $\xi=l_{u}(z)$. By using the identity
 \begin{align*}
 \frac{\xi \tilde f'(\xi)}{\tilde f(\xi)}=1+\xi\frac{d}{d\xi} \log \frac{\tilde f(\xi)}{\xi},
 \end{align*}
 we get from \eqref{Th-2.2} that
 \begin{align}\label{Th-2.3}
P(\xi)= \frac{\tilde f(\xi)}{\xi}=\exp\left(\int\limits_{0}^{\xi}\frac{p(t)-1}{t} {\rm d}t\right).
 \end{align}

 \noindent Now using \eqref{Th-2.2} and \eqref{Th-2.3} to \eqref{Th-2.1}, we deduce that
 \begin{align}\label{Th-2.4}
 	\frac{f(\xi)}{\xi}\left\lbrack\frac{\xi f'(\xi)}{f(\xi)}\right\rbrack^{\alpha}=\frac{\tilde f(\xi)}{\xi}=P(\xi).
 \end{align}

\noindent  Making the substitution
 \begin{align}\label{Th-2.5}
 \frac{f(\xi)}{\xi} = u^{\alpha}(\xi)
 \end{align}
 into \eqref{Th-2.4}, we obtain
 \begin{align}\label{Th-2.6}
 \frac{du}{d\xi}=-\frac{u}{\alpha \xi} + \frac{P^{1/\alpha}(\xi)}{\alpha \xi}.
 \end{align}

 \noindent Now from \eqref{Th-2.5} and \eqref{Th-2.6}, we get
 \begin{align*}
 f(\xi)= \left\lbrack \frac{1}{\alpha} \int_{0}^{\xi} P^{1/\alpha}(t) t^{1/\alpha-1}\, {\rm d}t \right\rbrack^{\alpha}=\left\lbrack \frac{1}{\alpha} \int_{0}^{\xi} \tilde f^{1/\alpha}(t) t^{-1}\, {\rm d}t \right\rbrack^{\alpha}.
 \end{align*}

 Conversely, let $\tilde F(z)=\frac{\tilde f(l_u(z))}{l_u(z)}z\in\mathcal{S}^*(\mathbb{B})$ and $F(z)=\frac{ f(l_u(z))}{l_u(z)}z$, where $f$ is given by \eqref{Th-2.0}. According to Lemma 2.1 \cite{Xu-He-Xu}, we have $\tilde f\in\mathcal{S}^*(\mathbb{U})$. Consequently \eqref{Th-2.2} holds, where $p\in\mathcal{P}$. Using \eqref{Th-2.0}, one can easily arrive at \eqref{Th-2.4}. Then by a simple computation, we can easily arrive at \eqref{Th-2.1}. Finally, we conclude that $F(z)=\frac{f(l_{u}(z))}{l_{u}(z)}z\in\mathcal{M}_{\alpha}(\mathbb{B})$.
\end{proof}

\section{{\bf Coefficient problems for $\alpha$-convex functions}}\label{Sec-3}

In the following result, we give estimates for the upper bound of the second Hankel determinant for $f\in\mathcal{M}_{\alpha}(\mathbb{B})$ when $\alpha\geq 0$.
\begin{theo}\label{Th-31}Let $f\in\mathcal{H}(\mathbb{B},\mathbb{C})$ with $f(0)=1$ and $F(z)=zf(z)\in\mathcal{M}_{\alpha}(\mathbb{B})$, $\alpha\geq 0$. Then for $z\in\mathbb{B}\backslash\{0\}$, we have
\begin{align}\label{Th31}
&\left|\frac{l_{z}(D^2F(0)(z^2))}{2!\|z\|^2}\frac{l_{z}(D^4F(0)(z^4))}{4!\|z\|^4}-\left(\frac{l_z(D^3F(0)(z^3))}{3!\|z\|^3}\right)^2\right|\\& \le
	\begin{cases}
		\dfrac{2(1 + 8\alpha + 3\alpha^2)}{(1 + 3\alpha)(2 + 15\alpha + 24\alpha^2 + 7\alpha^3)}, & \text{when } \alpha \in [0, 1], \\[10pt]
		\dfrac{3 + 35\alpha + 126\alpha^2 + 162\alpha^3 + 79\alpha^4 + 27\alpha^5}{(1 + 2\alpha)^2 (1 + 3\alpha)(3 + 26\alpha + 42\alpha^2 + 18\alpha^3 + 7\alpha^4)}, & \text{when } \alpha \in [1, \infty).\nonumber
	\end{cases}
\end{align}	
\end{theo}

\begin{proof}
By the given condition, we have $F\in\mathcal{M}_{\alpha}(\mathbb{B})$ and so \eqref{Lem1.1} holds. Let $z\in \mathbb{B}\backslash\{0\}$ be fixed and let $z_0=z/\|z\|$.
We consider the function $H:\mathbb{U}\to \mathbb{C}$ defined as in \eqref{Lem1.2}. Clearly $H\in\mathcal{H}(\mathbb{U})$ such that $H(0)=1$
and $\Re\{H(\zeta)\}>0$ for all $\zeta\in\mathbb{U}$. Assume that
\begin{align}\label{Th-3.1}
H(\zeta)=1+p_1\zeta+p_2\zeta^2+\cdots.
\end{align}

\noindent Now using \eqref{Lem1.14} and \eqref{Lem1.16} to \eqref{Lem1.6}, we get
\begin{align}\label{Th-3.2}
\displaystyle
	\frac{||\zeta z_0||}{l_{\zeta z_0}\left([DF(\zeta z_0)]^{-1}F(\zeta z_0)\right)}=\frac{1+\sum\limits_{k=1}^{\infty}\frac{k+1}{k!}D^k f(0)(z_0^k)\zeta^k}{1+\sum\limits_{k=1}^{\infty}\frac{1}{k!}D^k f(0)(z_0^k)\zeta^k}.
\end{align}

\noindent On the other hand, \eqref{Lem1.18} yields
\begin{align}\label{Th-3.3}
		\frac{l_{\zeta z_0}\left([D F(\zeta z_0)]^{-1}\phi(\zeta z_0)\right)}{||\zeta z_0||}=\frac{1+\sum\limits_{m=1}^{\infty}\frac{(m+1)^2}{m!}D^{m} f(0)(z^{m})\zeta^m}{1+\sum\limits_{k=1}^{\infty}\frac{k+1}{k!}D^k f(0)(z_0^k)\zeta^k}.
\end{align}

\noindent In view of \eqref{Lem1.3} and using \eqref{Th-3.2} and \eqref{Th-3.3}, we deduce from that
\begin{equation}\label{Th-3.4}
H(\zeta)=(1-\alpha)\frac{1+\sum\limits_{k=1}^{\infty}\frac{k+1}{k!}D^k f(0)(z_0^k)\zeta^k}{1+\sum\limits_{k=1}^{\infty}\frac{1}{k!}D^k f(0)(z_0^k)\zeta^k}+\alpha
\frac{1+\sum\limits_{k=1}^{\infty}\frac{(k+1)^2}{k!}D^{k} f(0)(z_0^k)\zeta^k}{1+\sum\limits_{k=1}^{\infty}\frac{k+1}{k!}D^k f(0)(z_0^k)\zeta^k}.
\end{equation}

\noindent Using \eqref{Th-3.1} and \eqref{Th-3.4}, we get
\begin{align}\label{Th-3.5}
&~\left(1+\sum\limits_{k=1}^{\infty}p_k\zeta^k\right)
\left(1+\sum\limits_{k=1}^{\infty}\frac{D^k f(0)(z_0^k)\zeta^k}{k!}\right)\left(1+\sum\limits_{k=1}^{\infty}\frac{k+1}{k!}D^k f(0)(z_0^k)\zeta^k\right)\nonumber
\\=&(1-\alpha)\left(1+\sum\limits_{k=1}^{\infty}\frac{k+1}{k!}D^k f(0)(z_0^k)\zeta^k\right)\left(1+\sum\limits_{k=1}^{\infty}\frac{k+1}{k!}D^k f(0)(z_0^k)\zeta^k\right)\nonumber
\\&+\alpha\left(1+\sum\limits_{k=1}^{\infty}\frac{1}{k!}D^k f(0)(z_0^k)\zeta^k\right)\left(1+\sum\limits_{k=1}^{\infty}\frac{(k+1)^2}{k!}D^{k} f(0)(z_0^k)\zeta^k\right).
\end{align}

\noindent Clearly from \eqref{Th-3.5}, we find that
\begin{align}\label{Th-3.6}	
Df(0)(z_0)=\frac{p_1}{\alpha+1},
\end{align}
\begin{align}\label{Th-3.7}
p_2+3p_1Df(0)(z_0)=(1+2\alpha)D^2f(0)(z_0^2)+2\left(Df(0)(z_0)\right)^2
\end{align}
and
\begin{align}\label{Th-3.8}
&p_3+3p_2Df(0)(z_0)+2p_1D^2f(0)(z_0^2)+2p_1\left(Df(0)(z_0)\right)^2\nonumber\\&=\left(\frac{1+3\alpha}{2}\right)D^3f(0)(z_0^3)+\left(\frac{7+\alpha}{2}\right)Df(0)(z_0)D^2f(0)(z_0^2).
\end{align}

\noindent Substituting the value of \eqref{Th-3.6} into \eqref{Th-3.7}, we obtain
\begin{align}\label{Th-3.9}
D^2f(0)(z_0^2)=\frac{(1+\alpha)^2p_2+(1+3\alpha)p_1^2}{(1+2\alpha)(1+\alpha)^2}.
\end{align}

\noindent Again using \eqref{Th-3.6}--\eqref{Th-3.9} into \eqref{Th-3.8}, we get
\begin{align}\label{Th-3.10}
\displaystyle
D^3f(0)(z_0^3)=\frac{2(1+\alpha)^3(1+2\alpha)p_3+(1+\alpha)^2(3+15\alpha)p_1p_2+(1+6\alpha+17\alpha^2)p_1^3}{(1+3\alpha)(1+2\alpha)(1+\alpha)^3}.
\end{align}

\noindent Since $F(z)=zf(z)$, by applying \eqref{Th-1y} into \eqref{Th-3.6}, \eqref{Th-3.9} and \eqref{Th-3.10} respectively, we obtain
\begin{align}\label{Th-3.11}
	&\frac{l_{z}(D^4F(0)(z^4))}{4!\|z\|^4}=\frac{2(1+\alpha)^3(1+2\alpha)p_3+(1+\alpha)^2(3+15\alpha)p_1p_2+(1+6\alpha+17\alpha^2)p_1^3}{6(1+3\alpha)(1+2\alpha)(1+\alpha)^3},
\end{align}
\begin{align}\label{Th-3.12}
	\frac{l_{z}(D^3F(0)(z^3))}{3!\|z\|^3}=\frac{(1+\alpha)^2p_2+(1+3\alpha)p_1^2}{2(1+2\alpha)(1+\alpha)^2},
\end{align}
and
\begin{align}\label{Th-3.13}
	\frac{l_z(D^2F(0)(z^2))}{2!\|z\|^2}=\frac{p_1}{\alpha+1}.
\end{align}

\noindent Using \eqref{Th-3.11}-\eqref{Th-3.13}, we find that
\begin{align*}
&\frac{l_{z}(D^2F(0)(z^2))}{2!\|z\|^2}\frac{l_{z}(D^4F(0)(z^4))}{4!\|z\|^4}-\left(\frac{l_z(D^3F(0)(z^3))}{3!\|z\|^3}\right)^2\\=&	
-\frac{(1 + 10\alpha + 13\alpha^2)p_1^4}{12(1 + \alpha)^3(1 + 2\alpha)^2(1 + 3\alpha)} + \frac{\alpha p_1^2 p_2}{2(1 + \alpha)(1 + 2\alpha)^2(1 + 3\alpha)} - \frac{p_2^2}{4(1 + 2\alpha)^2}\\& + \frac{p_1 p_3}{3(1 + \alpha)(1 + 3\alpha)}.
\end{align*}

\noindent Now proceeding in the same way as done in the proof of Theorem 1 \cite{Sokol-Thomas-2018}, we get the desired inequality \eqref{Th31}.
\end{proof}

As consequences of Theorem \ref{Th-31}, we obtain the following corollaries:
\begin{cor}\label{Cor-31}Let $f\in\mathcal{H}(\mathbb{B},\mathbb{C})$ with $f(0)=1$ and $F(z)=zf(z)\in\mathcal{M}_{\alpha}(\mathbb{B})$, $\alpha>0$. Then for $z\in\mathbb{B}\backslash\{0\}$, we have
\begin{align*}
\left|\frac{l_{z}(D^2F(0)(z^2))}{2!\|z\|^2}\right| \le \frac{2}{1+\alpha}.
\end{align*}
The inequality is sharp.	
\end{cor}
\begin{proof}[\bf Proof of the Corollary \ref{Cor-31}] It is clear from \eqref{Th-3.13} that
\begin{align*}
\frac{l_z(D^2F(0)(z^2))}{2!\|z\|^2}=\frac{p_1}{\alpha+1}.
\end{align*} 	

\noindent Since $|p_1|\leq 2$, we easily get the desired inequality.\vspace{2mm}

{\bf Sharpness part:} We have $F(z)=\frac{f(l_u(z))}{l_u(z)}z$, where
\begin{align*}
f(l_u(z))=\left\lbrack\frac{1}{\alpha}\int\limits_{0}^{l_u(z)}\frac{\xi^{1/\alpha}}{\xi(1-\xi)^{2/\alpha}}{\rm d}\xi \right\rbrack^{\alpha},\quad z\in\mathbb{B},\quad ||u||=1.
\end{align*}
Since $f\in\mathcal{M}_{\alpha}(\mathbb{U})$, by Proposition \ref{pro-1}, we have $F\in\mathcal{M}_{\alpha}(\mathbb{B})$.
Taking $z=ru\;(0<r<1)$, we get from \eqref{TheoB3} that
\begin{align*}
\frac{f(l_u(z))}{l_u(z)}=\left\lbrack {}_{2} F_{1}\left(\frac{1}{\alpha},\frac{2}{\alpha},\frac{1}{\alpha}+1;r\right)\right\rbrack^{\alpha},
\end{align*}
where ${}_{2} F_{1}$ is the Gaussian hypergeometric function. It is easy to obtain
\begin{align*}
\frac{f(l_u(z))}{l_u(z)}=1+(2/(1+\alpha))r+O(r^2).
\end{align*}

\noindent Clearly $\frac{D^{2}F(0)\left(z^{2}\right)}{2!}=(2/(1+\alpha))r z$. Since
$z=ru\;(0<r<1)$, we get
\begin{align*}
\left|\frac{l_z\left(D^{2}F(0)\left(z^{2}\right)\right)}{2!||z||^{}}\right|=\frac{2}{1+\alpha},\quad z\in\mathbb{B}\backslash\{0\}.
\end{align*}
\end{proof}

\begin{cor}\label{Cor-32}Let $f\in\mathcal{H}(\mathbb{B},\mathbb{C})$ with $f(0)=1$ and $F(z)=zf(z)\in\mathcal{M}_{0}(\mathbb{B})=\mathcal{S}^*(\mathbb{B})$. Then for $z\in\mathbb{B}\backslash\{0\}$, we have
\begin{align*}
\left|\frac{l_{z}(D^2F(0)(z^2))}{2!\|z\|^2}\frac{l_{z}(D^4F(0)(z^4))}{4!\|z\|^4}-\left(\frac{l_z(D^3F(0)(z^3))}{3!\|z\|^3}\right)^2\right|\leq 1.
\end{align*}
The inequality is sharp.	
\end{cor}
 To establish the sharpness of the bound of Corollary \ref{Cor-32}, let us consider the function $F$ defined by
\begin{align*}
	F(z)=zf(z)=\frac{z}{1-l_u(z)^2},\quad z\in\mathbb{B},\quad \|u\|=1.
\end{align*}

\noindent According to Lemma 2.2 \cite{Xu-He-Xu}, it is clear that $F\in\mathcal{S}^*(\mathbb{B})$. Note that $f(z)=\sum_{n=0}^{\infty}(l_u(z))^{2n}$, $z\in\mathbb{B}$.
Clearly \eqref{Th-1y} yields
\begin{align*}
	\frac{l_z(D^2F(0)(z^2))}{2!}=0,\;\;
	\frac{l_z(D^3F(0)(z^3))}{3!}=(l_u(z))^2\|z\|\;\;\text{and}\;\;
	\frac{l_z(D^4F(0)(z^4))}{4!}=0.
\end{align*}

\noindent Taking $z=ru\;(0<r<1)$, we get
\begin{align*}
	\frac{l_z(D^2F(0)(z^2))}{2!||z||^2}=0,\;\;
	\frac{l_z(D^3F(0)(z^3))}{3!||z||^3}=1\quad \text{and}\quad
	\frac{l_z(D^4F(0)(z^4))}{4!||z||^4}=0.
\end{align*}

\noindent Therefore,
\begin{align*}
	\left|\frac{l_{z}(D^2F(0)(z^2))}{2!\|z\|^2}\frac{l_{z}(D^4F(0)(z^4))}{4!\|z\|^4}-\left(\frac{l_z(D^3F(0)(z^3))}{3!\|z\|^3}\right)^2\right|=1.
\end{align*}

 \begin{cor}\label{Cor-33}Let $f\in\mathcal{H}(\mathbb{B},\mathbb{C})$ with $f(0)=1$ and $F(z)=zf(z)\in\mathcal{M}_{1}(\mathbb{B})=\mathcal{K}_B(\mathbb{B})$. Then for $z\in\mathbb{B}\backslash\{0\}$, we have
 \begin{align*}
 \left|\frac{l_{z}(D^2F(0)(z^2))}{2!\|z\|^2}\frac{l_{z}(D^4F(0)(z^4))}{4!\|z\|^4}-\left(\frac{l_z(D^3F(0)(z^3))}{3!\|z\|^3}\right)^2\right|\leq 1/8.
 \end{align*}
 The inequality is sharp.	
 \end{cor}

\noindent To establish the sharpness of the bound of Corollary \ref{Cor-33}, let us consider the function $F$ defined by
\begin{align*}
	F(z)=zf(z)=\frac{f_1(l_u(z))}{l_u(z)}z,\quad z\in\mathbb{B},\quad \|u\|=1,
\end{align*}
where
\begin{align*}
	f_1(l_u(z))=\frac{1}{\omega^2-\omega}\log\frac{1+\omega^2l_u(z)}{1+\omega l_u(z)},
\end{align*}
with $\omega=(1+i\sqrt{3})/2$. Straightforward calculation yields
\begin{align*}
	1+\frac{zf_1''(z)}{f_1'(z)}=p(z)=\frac{1-z^2}{1-z+z^2}.
\end{align*}

\noindent Note that
\begin{align*}
	\Re\{p(z)\}=\frac{(1-|z|^2)(1+|z|^2-\Re\{z\})}{|1-z+z^2|^2}.
\end{align*}

\noindent Clearly $\Re\{p(z)\}>0$ for $z\in\mathbb{U}$ and so $f_1\in\mathcal{K}(\mathbb{U})$. According to Corollary 2 \cite{Rop-Su}, $F\in\mathcal{K}_B(\mathbb{B})$. On the other hand, we have
\begin{align*}f_1(z)=z+\frac{z^2}{2}-\frac{z^4}{4}-\frac{z^5}{5}+\frac{z^7}{7}+\frac{z^8}{8}+\cdots.
\end{align*}

\noindent Therefore,
\begin{align*}
	f(z)=\frac{f_1(l_u(z))}{l_u(z)}=1+\frac{l_u(z)}{2}-\frac{l_u(z)^3}{4}-\frac{l_u(z)^4}{5}+\frac{l_u(z)^6}{7}+\frac{l_u(z)^7}{8}+\cdots
\end{align*}
for $z\in\mathbb{B}$ and $||u||=1$.
Clearly \eqref{Th-1y} yields
\begin{align*}
	\frac{l_z(D^2F(0)(z^2))}{2!}=\frac{l_u(z)||z||}{2},\;\;
	\frac{l_z(D^3F(0)(z^3))}{3!}=0\;\;\text{and}\;\;
	\frac{l_z(D^4F(0)(z^4))}{4!}=\frac{l_u(z)^3||z||}{4}.
\end{align*}

\noindent Taking $z=ru\;(0<r<1)$, we get
\begin{align*}
	\frac{l_z(D^2F(0)(z^2))}{2!||z||^2}=\frac{1}{2},\;\;
	\frac{l_z(D^3F(0)(z^3))}{3!||z||^3}=0\;\;\text{and}\;\;
	\frac{l_z(D^4F(0)(z^4))}{4!||z||^4}=\frac{1}{4}.
\end{align*}

\noindent Therefore,
\begin{align*}
	\left|\frac{l_{z}(D^2F(0)(z^2))}{2!\|z\|^2}\frac{l_{z}(D^4F(0)(z^4))}{4!\|z\|^4}-\left(\frac{l_z(D^3F(0)(z^3))}{3!\|z\|^3}\right)^2\right|=\frac{1}{8},
\end{align*}
which establishes the sharpness of the bound.\vspace{2mm}

By applying these techniques, we demonstrate that a generalized version of Theorem \ref{Th-31} can be established:
\begin{theo}\label{Th-32}Let $f\in\mathcal{H}(\mathbb{B},\mathbb{C})$ with $f(0)=1$ and $F(z)=zf(z)\in\mathcal{S}(\mathbb{B})$, $\alpha\in\mathbb{R}$ and $0\leq \beta\leq 1$. If
\begin{align}\label{Th-30}
\displaystyle
\Re\left\lbrace (1-\alpha)\frac{||z||}{l_z\left([DF(z)]^{-1}F(z)\right)} +\alpha\frac{l_z\left([DF(z)]^{-1}\phi(z)\right)}{||z||}\right\rbrace>\beta,
\end{align}
$z\in\mathbb{B}\backslash\{0\}$ and $l_z\in T(z)$, where $\phi(z)=D^2 F(z)(z^2)+DF(z) z$, then for $z\in\mathbb{B}\backslash\{0\}$, we have
\begin{align}\label{Th-32.2}
&\left|\frac{l_{z}(D^2F(0)(z^2))}{2!\|z\|^2}\frac{l_{z}(D^4F(0)(z^4))}{4!\|z\|^4}-\left(\frac{l_z(D^3F(0)(z^3))}{3!\|z\|^3}\right)^2\right|\\& \le
\begin{cases}
	(1 - \beta)^2, & \text{if } \alpha = \beta = 0, \\[8pt]
	\dfrac{(1 - \beta)^2 (S_1 + S_2)}{(\alpha + \beta)(1 + 3\alpha) S_3}, & \text{if } 0 < \alpha + 2\beta \le 1, \\[12pt]
	\dfrac{(1 - \beta)^2 (S_4 + S_5)}{(1 + 2\alpha)^2 (1 + 3\alpha) S_6}, & \text{if } \alpha + 2\beta \ge 1,\nonumber
\end{cases}
\end{align}	
where $S_1,\ldots,S_6$ are defined in Theorem E.
\end{theo}
\begin{proof}[\bf Proof of Theorem \ref{Th-32}]
By the given condition, we have $F\in\mathcal{M}_{\alpha}(\mathbb{B})$ and so \eqref{Lem1.1} holds. Let $z\in \mathbb{B}\backslash\{0\}$ be fixed and let $z_0=z/\|z\|$.
We consider the function $H:\mathbb{U}\to \mathbb{C}$ defined as in \eqref{Lem1.2}. Clearly $H\in\mathcal{H}(\mathbb{U})$ such that $H(0)=1$
and $\Re\{H(\zeta)\}>\beta$ for all $\zeta\in\mathbb{U}$. Let
\begin{align*}
\tilde H(\zeta)=\frac{H(\zeta)-\beta}{1-\beta}.
\end{align*}
Clearly $\tilde H\in\mathcal{H}(\mathbb{U})$ such that $\tilde H(0)=1$
and $\Re\{\tilde H(\zeta)\}>\beta$ for all $\zeta\in\mathbb{U}$.	
Assume that
\begin{align*}
\displaystyle
\tilde H(\zeta)=1+p_1\zeta+p_2\zeta^2+\ldots
\end{align*}

\noindent Then from above, we have
\begin{align}\label{Th-32.3}
\displaystyle
H(\zeta)=1+(1-\beta)p_1\zeta+(1-\beta)p_2\zeta^2+\cdots.
\end{align}

\noindent Using \eqref{Th-3.4} and \eqref{Th-32.3}, we get
\begin{align}\label{Th-32.4}
	&\left(1+(1-\beta)p_1\zeta+(1-\beta)p_2\zeta^2+(1-\beta)p_3\zeta^3+\cdots\right)
	\left(1+\sum\limits_{k=1}^{\infty}\frac{1}{k!}D^k f(0)(z_0^k)\zeta^k\right)\times\nonumber\\&\left(1+\sum\limits_{k=1}^{\infty}\frac{k+1}{k!}D^k f(0)(z_0^k)\zeta^k\right)\nonumber\\=&
	(1-\alpha)\left(1+\sum\limits_{k=1}^{\infty}\frac{k+1}{k!}D^k f(0)(z_0^k)\zeta^k\right)\left(1+\sum\limits_{k=1}^{\infty}\frac{k+1}{k!}D^k f(0)(z_0^k)\zeta^k\right)\nonumber\\&+\alpha\left(1+\sum\limits_{k=1}^{\infty}\frac{1}{k!}D^k f(0)(z_0^k)\zeta^k\right)\left(1+\sum\limits_{k=1}^{\infty}\frac{(k+1)^2}{k!}D^{k} f(0)(z_0^k)\zeta^k\right).
\end{align}

\noindent Clearly from \eqref{Th-32.4}, we find that
\begin{align}\label{Th-4.6}	
	Df(0)(z_0)=\frac{(1-\beta)p_1}{(1+\alpha)},
\end{align}
\begin{align}\label{Th-4.7}
	(1-\beta)\left(p_2+3p_1Df(0)(z_0)\right)=(1+2\alpha)D^2f(0)(z_0^2)+2\left(Df(0)(z_0)\right)^2
\end{align}
and
\begin{align}\label{Th-4.8}
	&(1-\beta)\left(p_3+3p_2Df(0)(z_0)+2p_1D^2f(0)(z_0^2)+2p_1\left(Df(0)(z_0)\right)^2\right)\nonumber\\&=\left(\frac{1+3\alpha}{2}\right)D^3f(0)(z_0^3)+\left(\frac{7+\alpha}{2}\right)Df(0)(z_0)D^2f(0)(z_0^2).
\end{align}

\noindent Substituting the value of \eqref{Th-4.6} into \eqref{Th-4.7}, we obtain
\begin{align}\label{Th-4.9}
	D^2f(0)(z_0^2)=\frac{(1-\beta)\left((1+\alpha)^2p_2+(1-\beta)(1+3\alpha)p_1^2\right)}{(1+2\alpha)(1+\alpha)^2}.
\end{align}

\noindent Again using \eqref{Th-4.6}--\eqref{Th-4.9} into \eqref{Th-4.8}, we get
\begin{align}\label{Th-4.10}
	\displaystyle
	D^3f(0)(z_0^3)=&\frac{(1-\beta)}{(1+3\alpha)(1+2\alpha)(1+\alpha)^3}\big(2(1+\alpha)^3(1+2\alpha)p_3\\&+(1-\beta)(1+\alpha)^2(3+15\alpha)p_1p_2+(1-\beta)^2(1+6\alpha+17\alpha^2)p_1^3\big)\nonumber.
\end{align}

\noindent Since $F(z)=zf(z)$, by applying \eqref{Th-1y} into \eqref{Th-4.6}, \eqref{Th-4.9} and \eqref{Th-4.10} respectively, we obtain
\begin{align}\label{Th-4.11}
	\frac{l_{z}(D^4F(0)(z^4))}{4!\|z\|^4}=&\frac{(1-\beta)}{6(1+3\alpha)(1+2\alpha)(1+\alpha)^3}\big(2(1+\alpha)^3(1+2\alpha)p_3\\&+(1-\beta)(1+\alpha)^2(3+15\alpha)p_1p_2+(1-\beta)^2(1+6\alpha+17\alpha^2)p_1^3\big)\nonumber.
\end{align}
\begin{align}\label{Th-4.12}
	\frac{l_{z}(D^3F(0)(z^3))}{3!\|z\|^3}=\frac{(1-\beta)\left((1+\alpha)^2p_2+(1-\beta)(1+3\alpha)p_1^2\right)}{2(1+2\alpha)(1+\alpha)^2},
\end{align}
and
\begin{align}\label{Th-4.13}
	\frac{l_z(D^2F(0)(z^2))}{2!\|z\|^2}=\frac{(1-\beta)p_1}{(1+\alpha)}.
\end{align}

Using \eqref{Th-4.11}-\eqref{Th-4.13}, we find that
\begin{align*}
	&\frac{l_{z}(D^2F(0)(z^2))}{2!\|z\|^2}\frac{l_{z}(D^4F(0)(z^4))}{4!\|z\|^4}-\left(\frac{l_z(D^3F(0)(z^3))}{3!\|z\|^3}\right)^2\\=&	
		(1-\beta)^2\bigg[\frac{\left(2(1-\beta)^2(1+2\alpha)(1+6\alpha+17\alpha^2)-3(1+3\alpha)^2\right)}{12(1+3\alpha)(1+2\alpha)^2(1+\alpha)^4}p_1^4\\&+\frac{\left((1-\beta)(1+5\alpha)(1+2\alpha)-(1+3\alpha)^2 \right)p_1^2p_2}{2(1+3\alpha)(1+2\alpha)^2(1+\alpha)}-\frac{p_2^2}{4(1+2\alpha)^2}+\frac{p_1p_3}{3(1+3\alpha)(1+\alpha)}\bigg].
\end{align*}
Now proceeding in the same way as done in the proof of Theorem 2 \cite{Sokol-Thomas-2018}, we get the desired inequality \eqref{Th31}.
\end{proof}
\begin{rem}For $\alpha=0$, \eqref{Th-30} becomes $\mathcal{S}^*_{\beta}(\mathbb{B})$ and so from \eqref{Th-32.2}, we have
\begin{align*}
\left|\frac{l_{z}(D^2F(0)(z^2))}{2!\|z\|^2}\frac{l_{z}(D^4F(0)(z^4))}{4!\|z\|^4}-\left(\frac{l_z(D^3F(0)(z^3))}{3!\|z\|^3}\right)^2\right|\leq (1-\beta)^2,\quad z\in\mathbb{B}\backslash\{0\},
\end{align*} 	
which extends the result due to Lee et al. \cite{Lee-2013} to complex Banach spaces.
For $\alpha=1$, \eqref{Th-30} becomes $\mathcal{K}^{\beta}_{B}(\mathbb{B})$.
Consequently, from \eqref{Th-32.2}, we obtain
 \begin{align*}
 	\left|\frac{l_{z}(D^2F(0)(z^2))}{2!\|z\|^2}\frac{l_{z}(D^4F(0)(z^4))}{4!\|z\|^4}-\left(\frac{l_z(D^3F(0)(z^3))}{3!\|z\|^3}\right)^2\right|\leq (1 - \beta)^2 \frac{36 - 36\beta + 17\beta^2}{144(2 - 2\beta + \beta^2)},
 \end{align*}
$z\in\mathbb{B}\backslash\{0\}$, which is not sharp. This extends the result of Krishna and Ramreddy \cite{Krishna-Ramreddy-2012} to complex Banach spaces.
For $\beta=0$, the bounds \eqref{Th-32.2} become \eqref{Th31} in Theorem \ref{Th-31}.
\end{rem}

\vspace{5mm}

\noindent\textbf{Conflict of interest:} The authors declare that there is no conflict  of interest regarding the publication of this paper.\vspace{1.2mm}

\noindent {\bf Funding:} Not Applicable.\vspace{1.2mm}

\noindent\textbf{Data availability statement:}  Data sharing not applicable to this article as no datasets were generated or analysed during the current study.\vspace{1.2mm}


\end{document}